\documentclass[12pt]{amsart}
\usepackage{amssymb}
\usepackage{amsmath}
\usepackage{bbm}
\usepackage[dvips]{graphicx}
\usepackage{hyperref}
\usepackage[capitalize,nameinlink,noabbrev,nosort]{cleveref}
\usepackage[dvipsnames]{xcolor}
\usepackage{todonotes}
\usepackage{enumitem}

\numberwithin{equation}{section}

\newtheorem{thm}{Theorem}[section]
\newtheorem{prop}[thm]{Proposition}
\newtheorem{cor}[thm]{Corollary}
\newtheorem{eg}[thm]{Example}
\newtheorem{lem}[thm]{Lemma}
\theoremstyle{definition}
\newtheorem{defn}[thm]{Definition}
\newtheorem{rem}[thm]{Remark}

\newcommand{\tcr}[1]{\textcolor{red}{#1}}
\newcommand{\bZ}{\mathbb{Z}}

\newcommand{\ds}{\displaystyle}
\newcommand{\1}{\bullet}
\newcommand{\2}{{\scriptstyle \Box}}
\newcommand{\0}{\cdot}
\newcommand{\aver}[1]{\langle #1 \rangle}
\DeclareMathOperator{\wt}{wt}

\title
[A disordered two-species ASEP on a torus]
{Combinatorics of a disordered two-species ASEP on a torus}

\author{Arvind Ayyer}
\address{Arvind Ayyer, Department of Mathematics, 
Indian Institute of Science, Bangalore  560012, India.}
\email{arvind@iisc.ac.in}

\author{Philippe Nadeau}
\address{Philippe Nadeau, Univ Lyon, CNRS, Universit\'e Claude Bernard Lyon 1, UMR 5208, Institut Camille Jordan, 43 blvd. du 11 novembre 1918, F-69622 Villeurbanne cedex, France}
\email{nadeau@math.univ-lyon1.fr}

\date{\today}

\begin{document}

\begin{abstract}
We define a new disordered asymmetric simple exclusion process (ASEP) with two species of particles, first-class particles labelled $\1$ and second-class particles labelled $\2$, on a two-dimensional toroidal lattice.  
The dynamics is controlled by particles labelled $\1$, which only move horizontally, with forward and backward hopping rates $p_i$ and $q_i$ respectively if the $\1$ is on row $i$. 
The motion of particles labelled $\2$ depends on the
relative position of these with respect to $\1$'s, and can be both
horizontal and vertical. We show that the stationary weight of any configuration is proportional to a monomial in the $p_i$'s and $q_i$'s.
Our process projects to the disordered ASEP on a ring, and so explains combinatorially the stationary distribution of the latter first derived by Evans (Europhysics Letters, 1996). 
We compute the partition function, as well as densities and currents of $\1$'s and $\2$'s in the stationary state.
We observe a novel mechanism we call the \emph{Scott Russell phenomenon}: the current of $\2$'s in the vertical direction is the same as that of $\1$'s in the horizontal direction.
\end{abstract}


\maketitle

\section{Introduction}
\label{sec:intro}

The asymmetric simple exclusion process (ASEP) is an important model in nonequilibrium statistical physics. Over the last few decades, the one-dimensional ASEP on a finite one-dimensional lattice with open boundaries~\cite{dehp} has been intensively studied by mathematicians due to the simple yet nontrivial combinatorial structure of its stationary distribution; see for example~\cite{duchi-schaeffer-2005,corteel-williams-2010}. The stationary distribution of the one-dimensional ASEP with periodic boundary conditions is uniform, but the story gets interesting if there are two species of particles, one faster and one slower. 
These are called first and second class particles respectively, and were first considered in the study of shock measures in the single species ASEP on $\mathbb{Z}$ in~\cite{andjel-bramson-liggett-1988}; see~\cite[Part III, Chapter 2]{liggett-sis-1999} for more details.
In this case too, the stationary distribution has an elegant combinatorial structure~\cite{djls-1993}. The combinatorics of the closed two-species has also been understood from different points of view~\cite{angel-2006,FM07,ayyer-linusson-2014, mandelshtam-2020, martin-2020}.
Another combinatorial generalization of the one-dimensional single-species ASEP is where the particles have disordered rates, i.e., the $k$'th particle hops forward (resp. backward) with rate $p_k$ (resp. $q_k$). 
This was first studied by Spitzer~\cite[Section 5a]{spitzer-1970} in the symmetric case ($p_k = q_k$) and later generalised by Evans~\cite{evans-1996}. It is the combinatorial structure of this disordered ASEP that we will unravel here.

We present an exact solution of a two-dimensional exclusion process with closed boundaries (i.e. on a discrete $L \times n$ torus) with two kinds of particles. To the best of our knowledge, when there are multiple species of particles and the rates are disordered, no formulas for finite systems exist in the literature. This is the first two-dimensional disordered exclusion process whose stationary distribution is understood exactly.

The \emph{first-class} particles are denoted $\1$, and the \emph{second-class} particles are denoted $\2$. There is one first-class particle per row and these only move horizontally, the particle on the $k$'th row moving forward with rate $p_k$ and backward with rate $q_k$. However, these particles dictate the motion of the second-class particles which move both horizontally and vertically; see \cref{sec:torus} for the precise definition. The process can also be viewed isomorphically as a one-dimensional multispecies ASEP (\cref{sub:colored}) by projection. 

More interestingly, it can also be formulated as a process on  set partitions (\cref{sub:partitions}) with $n$ blocks on $L$ elements, with some specific marking.

For this two-dimensional ASEP,
we give an explicit formula for the stationary distribution (\cref{thm:ssrefined}) and the nonequilibrium partition function (\cref{thm:pf}). In particular, we show that the stationary probability of any configuration is proportional to a monomial in the $p_k$'s and $q_k$'s. This two-dimensional ASEP projects to the disordered one-dimensional ASEP studied by Evans (\cref{prop:lump}), and this allows us to give a combinatorial formula for the stationary distribution of the latter (\cref{cor:ss}).

It turns out that the two-dimensional ASEP is interesting in its own right for several reasons. 

For two special cases, (i) $p_i = q_i$ for all $i$ and (ii) $q_i = 0$ for all $i$, we find that the partition function is a symmetric polynomial in the $p_i$'s; see \cref{prop:symmetric_particles} and \cref{prop:Z_totally_asymmetric}.
We give explicit formulas for the densities (i.e. the occupation probabilities in the stationary distribution) of both $\1$'s and $\2$'s. 
We then calculate the currents for both $\1$'s and $\2$'s across a given horizontal edge. 
Since the $\2$'s move nonlocally, we consider their horizontal current between any two adjacent columns as well as their vertical current between any two adjacent rows.

We find a remarkable coincidence, that the total current of $\1$'s in the horizontal direction in the $j$'th row is identical to that of the $\2$'s in the vertical direction between the $(j-1)$'th and $j$'th row (\cref{cor:curr1} and \cref{cor:curr2-v}). 
The fact that these two are the same does not follow from the dynamics.
We dub this the \emph{Scott Russell phenomenon}, named after the Scottish engineer (John) Scott Russell, the eponym for the \emph{linkage} which translates linear motion in one direction to that in a perpendicular direction. See \cref{fig:linkage} for an illustration of the linkage.
This is a manifestly two-dimensional occurrence and the reason why it is crucial to view this as a process on the torus rather than as a multispecies one-dimensional process or a process on set partitions.

\begin{center}
\begin{figure}[!ht]
\includegraphics[width=0.7\textwidth]{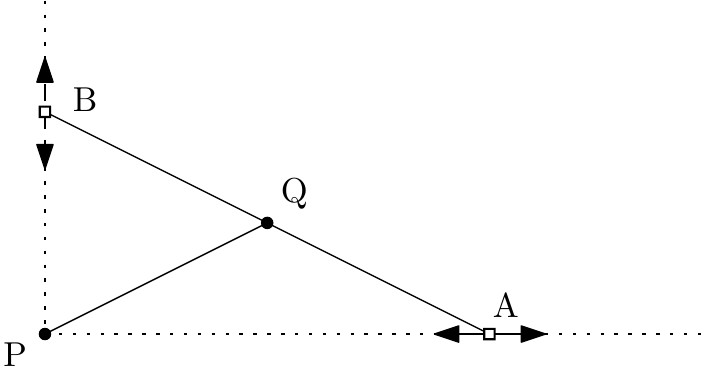}
\caption{A cartoon of the Scott Russell linkage. Here, PQ and AB are rigid bars. P is fixed and Q is a hinge. A and B are forced to move along the lines shown. When A is moved horizontally, B moves vertically at the same speed.}
\label{fig:linkage}
\end{figure}
\end{center}

Our two-dimensional process also sheds some light on the combinatorial structure underlying previous work by the first author~\cite{ayyer-2020} in which only one particle, known as a tracer, moves asymmetrically with forward and backward rates $p$ and $q$ respectively. The other particles move symmetrically with rate $1$. This also explains certain simplifications that occur in the recent work by Lobaskin and Evans~\cite{lobaskin-evans-2020} where they study a model with many totally asymmetric tracers (i.e. $q_i=0$ for all $i$).

The plan of the rest of the article is as follows. In \cref{sec:torus}, we define the two-dimensional model on the torus and explain how it can be interpreted as a one-dimensional multispecies exclusion process. We explain the projection to the inhomogeneous ASEP on the ring in \cref{sec:onedim}. We compute the stationary distribution and the partition function in \cref{sec:statdist}. 
The special case where some particles move totally asymmetrically is dealt with in \cref{sec:qi_zero}.
Finally, the densities and currents are derived and the Scott Russell phenomenon is explained in \cref{sec:denscurr}.

\section{The two-dimensional model on the torus}
\label{sec:torus}

We now define the exclusion process on a discrete $L \times n$ torus $\bZ/L\bZ\times \bZ/n\bZ$  with particles of two types and vacancies.
As mentioned earlier, we will denote first class particles by $\1$, second class particles by $\2$, and vacancies by $\0$.

\subsection{State space}
\label{sub:state_space}

\begin{defn}
\label{config-defi}
Let $\mathcal{A}_{L,n}$ consist of configurations 
$A \equiv (A_{i,j})_{\substack{1 \leq i \leq L \\ 1 \leq j \leq n}}$ where $A_{i,j}\in\{\0,\1,\2\}$ such that:
\begin{itemize}

\item Each row contains exactly one $\1$.

\item Each column contains exactly one particle (either $\1$ or $\2$).

\item The column indices of $\1$'s read from left to right form a cyclically increasing sequence, i.e. a sequence of integers for which a cyclic permutation exists transforming it to an increasing sequence.

\end{itemize}
\end{defn}

\begin{center}
\begin{figure}[!ht]
\includegraphics[width=0.9\textwidth]{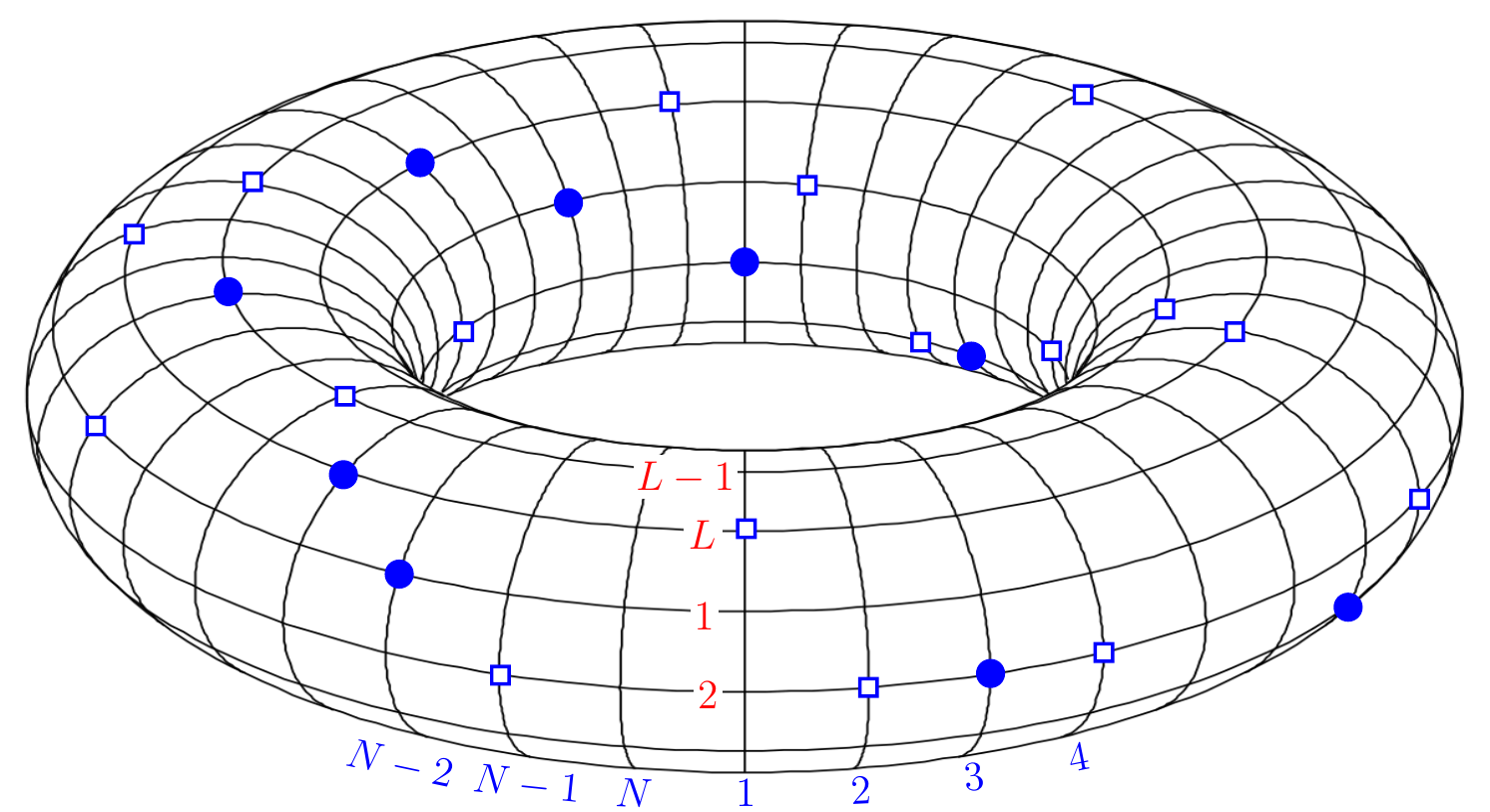}
\caption{An illustration of a configuration in $\mathcal{A}_{L,n}$.}
\label{fig:multiline_on_torus}
\end{figure}
\end{center}

Such a configuration is illustrated in Figure~\ref{fig:multiline_on_torus}; note that certain particles are hidden from view. For convenience, we will represent configurations $A \in \mathcal{A}_{L,n}$ can be written as arrays $A \equiv (A_{i,j})_{1 \leq i \leq L, 1 \leq j \leq n}$, keeping in mind that this is actually a torus so that rows and columns ``wrap around'' horizontally and vertically; see Figure~\ref{fig:small_example}, top.

It will turn out, because of the horizontal translation invariance of the dynamics, that it suffices to focus attention to configurations in $A \in \mathcal{A}_{L,n}$ such that $A_{1,1} = \1$. We will denote the set of such configurations by $\mathcal{A}'_{L,n}$.
With that normalization, note that the third condition in \cref{config-defi} becomes that the column indices of $\1$'s in $A$ form a strictly increasing sequence.
We call such configurations {\em restricted configurations}. For example, the set of such restricted configurations $\mathcal{A}'_{4,2}$ is depicted in Figure~\ref{fig:small_example}.

\begin{figure}[!ht]
\includegraphics[width=\textwidth]{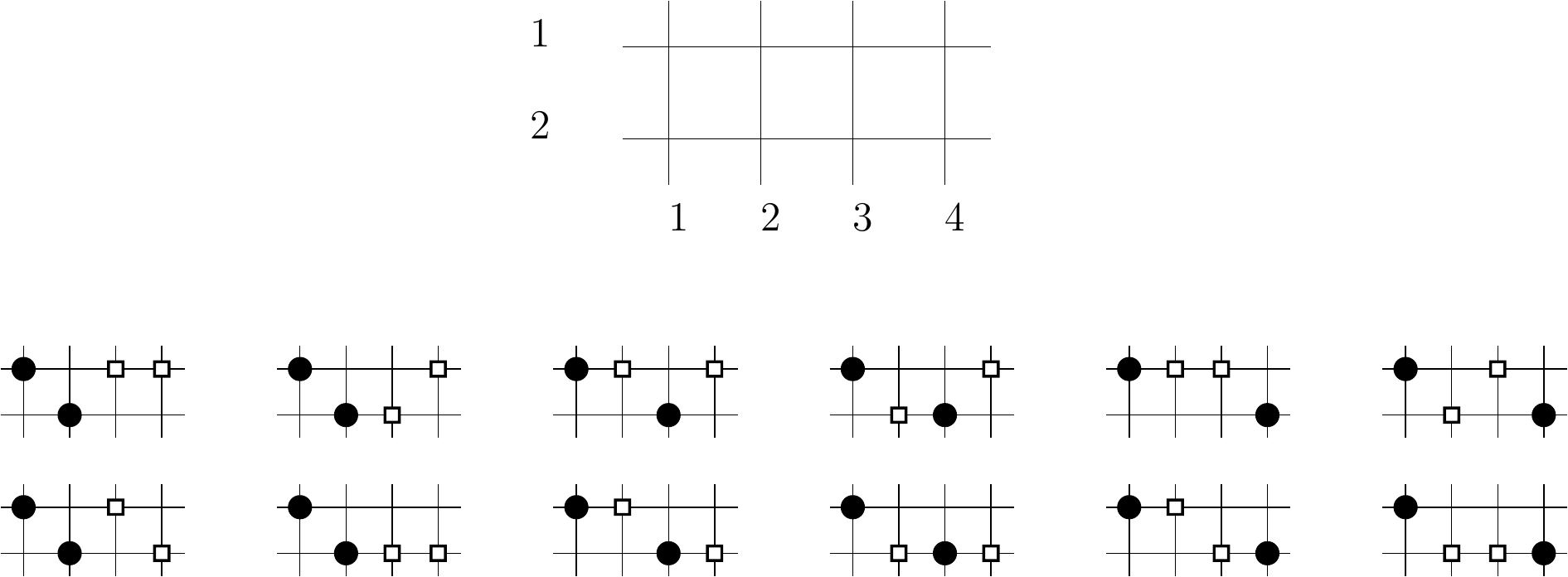}
\caption{Torus for $L=2$ and $N=4$, with all 12 of its restricted configurations.\label{fig:small_example}}
\end{figure}

\subsection{Dynamics}
\label{sub:dynamics}

The dynamics is a continuous-time Markov chain with the following transitions and rates.
Transitions are always initiated by particles of type $\1$. In fact, the vertical projection of the particles of type $\1$ follow an exclusion process on a one dimensional torus, cf. \cref{sec:onedim}.
We focus on the $\1$ in the $k$\textsuperscript{th} row so $A_{k,j}=\1$ for a unique index $j$. There are four types of transitions.

The first two are forward transitions. Now by definition $A_{k',j+1}\neq 0$ for a unique value of $k'$. For a forward transition to occur, we require $A_{k',j+1}=\2$, which we now assume. We distinguish two cases:
\begin{enumerate}
\item 
\label{it:transp1}
If $k'\neq k$, then we have the transition with rate $p_k$
\begin{center}
\includegraphics[scale=0.8]{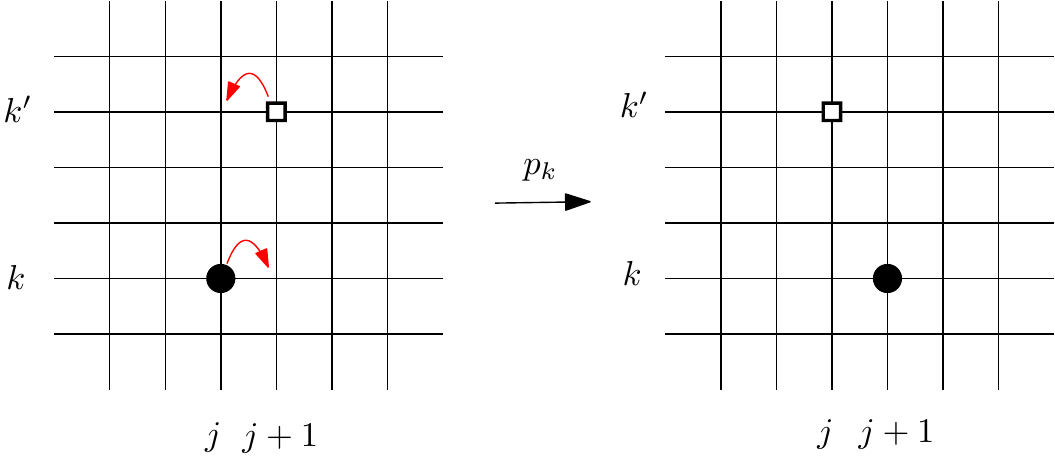}
\end{center}
That is, the new configuration $B$ satisfies $B_{k,j+1}=\1$, $B_{k',j}=\2$,  while the other columns are the same as in $A$.

\item 
\label{it:transp2}
If $k'=k$, then we have the transition with rate $p_k$
\begin{center}
\includegraphics[scale=0.8]{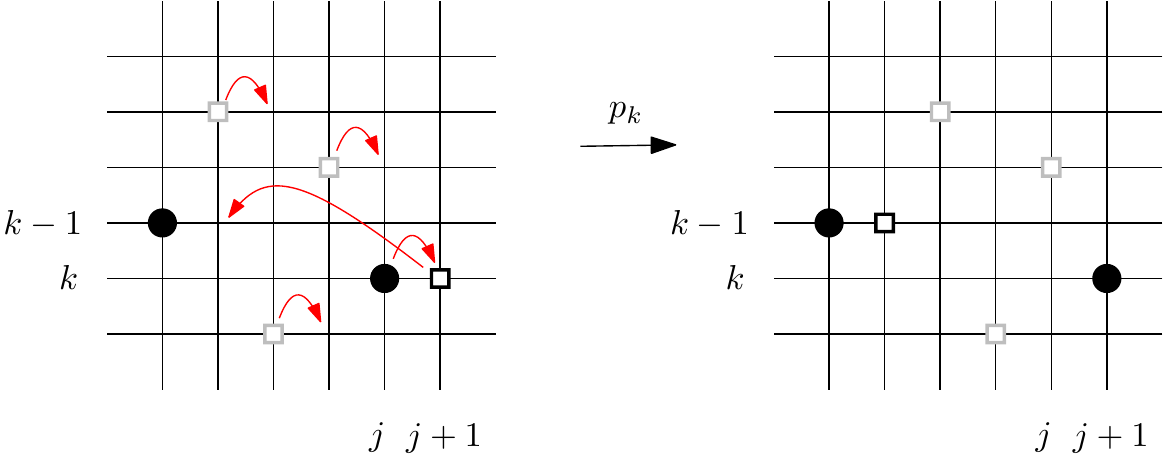}
\end{center}
That is, the new configuration $B$ satisfies $B_{i,l}=A_{i,l+1}$ for $l \in \{j'+1,\dots,j+1\}$ and any $i$, $B_{k-1,j'}=\2$, while the other columns are the same as in $A$. Here $j'$ is given by $A_{k-1,j'-1}=\1$.

We now consider backward transitions, which are defined in complete analogy, and so we illustrate them succinctly. By definition $A_{k'',j-1}\neq 0$ for a unique value of $k''$. For a backward transition to occur, we need to have $A_{k'',j-1}=\2$, which we again assume. As before, we have two cases:

\item 
\label{it:transq1}
If $k''\neq k$, then we have the transition with rate $q_k$
\begin{center}
\includegraphics[scale=0.8]{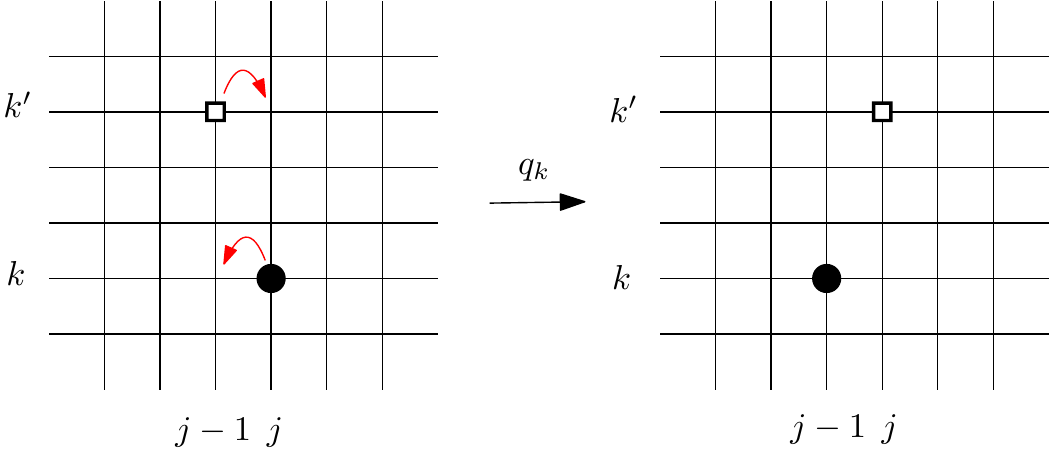}
\end{center}

\item 
\label{it:transq2}
If $k''= k$, then we have the transition with rate $q_k$
\begin{center}
\includegraphics[scale=0.8]{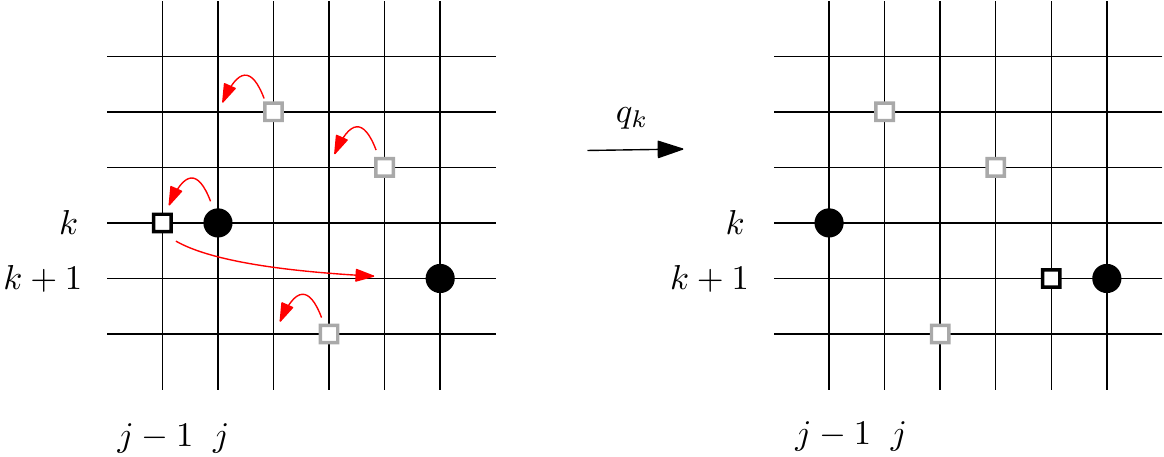}
\end{center}
\end{enumerate}

Note that \cref{it:transq1,it:transq2} are reversed versions of \cref{it:transp1,it:transp2} respectively. 
However, even when $p_k = q_k$ for all $k$, the dynamics is not reversible since transitions of types \cref{it:transp2} and \cref{it:transq2} are not inverses of each other.

From the general theory of Markov processes~\cite{norris-1998}, a continuous-time Markov chain is completely determined by its \emph{(column-stochastic) generator}. Recall that a generator is a matrix indexed by the configuration space whose $(i,j)$'th entry is equal to the transition rate from state $i$ to state $j$ if $i \neq j$ and whose diagonal entries are chosen such that column sums are zero.
The stationary probabilities are then given by the entries of the right null-eigenvector of the generator.

We make a few remarks about the symmetries of this dynamics:

\begin{rem}
\label{rem:dynamics}
\begin{enumerate}
\item The dynamics is invariant with respect to horizontal translation. Indeed, the value of $j$ does not modify the dynamics.
\item  Vertical translation modifies the rates by shifting $p_k\mapsto p_{k+1}$ and $q_k\mapsto q_{k+1}$.
\item  Forward and backward rules are directly related as follows: one can go from the first ones to the second ones, and vice versa, by simultaneous reflections along both coordinate axes (or equivalently, a rotation by $\pi$) together with an exchange of $p_k$ with $q_k$ for all $k$.
\end{enumerate}
\end{rem}

\subsection{Reformulation 1: Colored one-dimensional exclusion process.}
\label{sub:colored}

We consider now a second $(2n)$-species exclusion process (without vacancies) on the one-dimensional ring $\bZ/L\bZ$, with $L\geq n$. As will be quite evident from its definition, it is simply a more compact encoding of the previous process.

The $2n$ particles are labelled $\1_1,\ldots,\1_n$, ${\2}_1,\dots,{\2}_n$, and their indices will always be considered modulo $n$. The configurations can be naturally considered as words $w_1\ldots w_L$ where each $w_i$ is one of the $2n$ particles.

\begin{defn}[$\Omega_{L,n}$] 
The state space $\Omega_{L,n}$ consists of configurations with exactly one particle of each type $\1_1,\ldots,\1_n$ occurring cyclically in that order. The remaining $L-n$ positions are occupied by the particles ${\2}_1,\dots,{\2}_n$ and each of these can occur arbitrarily many times.
\end{defn}

Let $\Omega'_{L,n}\subset \Omega_{L,n}$ be the subset of restricted configurations $w$ defined by $w_1 = \1_1$. For example, the restricted configurations in $\Omega'_{4,2}$ are
\begin{multline}
\label{configs-42'}
\{\1_1 \1_2 {\2}_1 {\2}_1, \1_1 \1_2 {\2}_1 {\2}_2, \1_1 \1_2 {\2}_2 {\2}_1, \1_1 \1_2 {\2}_2 {\2}_2, 
\1_1 {\2}_1 \1_2 {\2}_1, \1_1 {\2}_1 \1_2 {\2}_2, \\
\1_1 {\2}_2 \1_2 {\2}_1, \1_1 {\2}_2 \1_2 {\2}_2, 
\1_1 {\2}_1 {\2}_1 \1_2, \1_1 {\2}_1 {\2}_2 \1_2, \1_1 {\2}_2 {\2}_1 \1_2, \1_1 {\2}_2 {\2}_2 \1_2 \}.
\end{multline}

There are clearly $\binom{L-1}{n-1}$ to choose the locations of the particles $\1_2,\1_3, \allowbreak \ldots,\1_n$. This leaves the remaining $L-n$ positions for the particles of the form $\2_i$, where $i$ can be chosen arbitrarily. It follows
\begin{align}
|\Omega'_{L,n}| =\frac{1}{L}|\Omega_{L,n}|=\binom{L-1}{n-1} n^{L-n}.
\end{align}

The transitions in $\Omega_{L,n}$ are given by the following rules for $1 \leq k \leq n$:
\begin{align}
\label{transp1}
\cdots \1_k| {\2}_i \cdots {\overset{p_k}{\longrightarrow}}&\quad \cdots {\2}_i | \1_k\cdots & \hspace*{-1.5cm} \text{ if }i\neq k,\\
\label{transp2}
 \cdots \1_{k-1} C \1_k|{\2}_k \cdots {\overset{p_k}{\longrightarrow}}&\quad \cdots \1_{k-1}{\2}_{k-1} C |\1_k\cdots,\\
\label{transq1}
\cdots  {\2}_i|\1_k\cdots {\overset{q_k}{\longrightarrow}}&\quad \cdots \1_k | {\2}_i\cdots & \hspace*{-1.5cm} \text{ if }i\neq k,\\
\label{transq2}
 \cdots  {\2}_k |\1_k C \1_{k+1}\cdots {\overset{q_k}{\longrightarrow}}&\quad \cdots \1_k | C {\2}_{k+1}\1_{k+1}\cdots,
\end{align}  
where we have placed a vertical divider to mark the location of the transition and $C$ is a (possibly empty) block which contains particles in the set $\{{\2}_1,\dots,{\2}_n\}$ between successive $\1$'s. Note that the integers $k+1$ and $k-1$ have to be interpreted modulo $n$ as was mentioned above.

\begin{eg}
\label{eg:trans}
Consider the configuration \[\tau = \1_1 {\2}_3{\2}_3{\2}_4\1_2 {\2}_2\1_3 {\2}_3\1_4 {\2}_1 \in \Omega'_{10,4}.\] The outgoing transitions from $\tau$ are 
\begin{align*}
\tcr{{\2}_3 \1_1} {\2}_3{\2}_4\1_2 {\2}_2\1_3 {\2}_3\1_4 {\2}_1 &\quad \text{with rate} \quad p_1, \\
\1_1 \tcr{{\2}_1} {\2}_3{\2}_3{\2}_4\tcr{\1_2} \1_3 {\2}_3\1_4 {\2}_1 &\quad \text{with rate} \quad p_2, \\
\1_1 {\2}_3{\2}_3{\2}_4\1_2 \tcr{{\2}_2} {\2}_2\tcr{\1_3} \1_4 {\2}_1 &\quad \text{with rate} \quad p_3, \\
\1_1 {\2}_3{\2}_3{\2}_4\1_2 {\2}_2\1_3 {\2}_3 \tcr{{\2}_1} \tcr{\1_4} & \quad \text{with rate} \quad p_4, \\
{\2}_3{\2}_3{\2}_4 \tcr{{\2}_2} \1_2 {\2}_2\1_3 {\2}_3\1_4 \tcr{\1_1} &\quad \text{with rate} \quad q_1, \\
\1_1 {\2}_3{\2}_3 \tcr{\1_2 {\2}_4} {\2}_2 \1_3 {\2}_3\1_4 {\2}_1 &\quad \text{with rate} \quad q_2, \\
\1_1 {\2}_3{\2}_3{\2}_4\1_2 \tcr{\1_3 {\2}_2} {\2}_3\1_4 {\2}_1 &\quad \text{with rate} \quad q_3, \\
\1_1 {\2}_3{\2}_3{\2}_4\1_2 {\2}_2\1_3 \tcr{\1_4 {\2}_3} {\2}_1 &\quad \text{with rate} \quad q_4.
\end{align*}
\end{eg}

{\bf Isomorphism between $\mathcal{A}_{L,n}$ and $\Omega_{L,n}$:}  
There is a simple  bijection between the two sets $\mathcal{A}_{L,n}$ and $\Omega_{L,n}$: If $A\in \mathcal{A}_{L,n}$, then for any $i=1,\ldots,L$, there is a unique $k_i$ such that $A_{k_i,i}\neq 0$, since each column contains exactly one particle. If  $A_{k_i,i}=\1$, define $w_i=\1_{k_i}$, and if $A_{k_i,i}=\2$ then define $w_i=\2_{k_i}$. The resulting word $w_1w_2\cdots w_L$ is the desired configuration in $\Omega_{L,n}$.

This can be simply understood as projecting while recording the labels of the rows as indices. See \cref{fig:multiline} for an illustration. Also, note that configurations in \cref{fig:small_example} project to those in \eqref{configs-42'}, in their given oredering.

\begin{center}
\begin{figure}[!ht]
\includegraphics[width=0.7\textwidth]{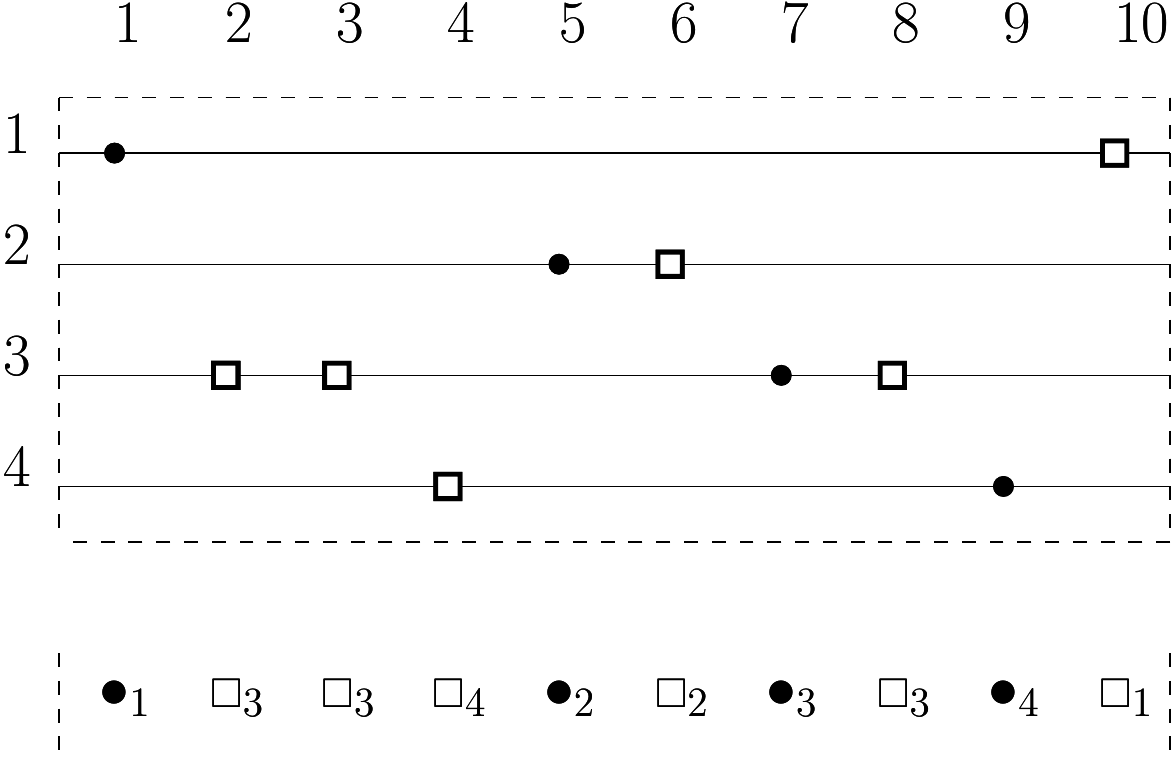}
\caption{
The configuration in $\mathcal{A}'_{10,4}$  (top) corresponding to the configuration $\tau \in \Omega'_{10,4}$ in \cref{eg:trans} (bottom), according to \cref{prop:isom}.
\label{fig:multiline}
}
\end{figure}
\end{center}

\begin{prop}
\label{prop:isom}
For arbitrary rates $p_k$ and $q_k$, the correspondence above is an isomorphism between the exclusion processes on $\mathcal{A}_{L,n}$ and $\Omega_{L,n}$.
\end{prop}

\begin{proof}
The correspondence is clearly bijective. The isomorphism between the two exclusion processes is established by verifying that the transitions \cref{it:transp1,it:transp2,it:transq1,it:transq2}
in $\mathcal{A}_{L,n}$
match exactly with those in 
\cref{transp1,transp2,transq1,transq2} 
in $\Omega'_{L,n}$. This is a simple inspection:
for example, \cref{it:transp1} corresponds to $\1_k \2_{k'} \to \2_{k'} \1 k$ where $k' \neq k$. The others follow similarly.
\end{proof}

\subsection{Reformulation 2: Marked set partitions}
\label{sub:partitions}

While the bijection in the last section is essentially straightforward, we now present another one which reveals some of the combinatorics of the model. Indeed
the configurations in $\mathcal{A}_{L,n}$ can be represented as certain (marked) set partitions, as follows:

 Let $A\in\mathcal{A}_{L,n}$. For $i\in\{1,\ldots,L\}$, let $B_i$ be the set of $j\in\{1,\ldots,n\}$ such that $A_{i,j}\neq 0$. That is, $B_i$ is the set of column indices of particles on row $i$ in $A$. We let $b_i\in B_i$ be the unique position such that $A_{i,b_i}=1$. Equivalently, in the model $\Omega_{L,n}$, $B_i$ is the set of positions of all particles $\2_i$ together with the position $b_i$ of the particle $\1_i$.

For example, the configuration in Figure~\ref{fig:multiline} corresponds to the following subsets with the elements $b_i$  underlined:
\[
(B_1,\dots,B_4) = (\{\underline{1},10\}, \{\underline{5},6\}, \{2,3,\underline{7},8\}, \{4,\underline{9}\}).
\]

  By the first two conditions in Definition~\ref{config-defi}, $B_i$, $i=1,\ldots,L$, form an ordered set partition of $\{1,\ldots,n\}$: that is that the $B_i$ are nonempty, disjoint, and their union is $\{1,\ldots,n\}$. Moreover, the third condition says that the $b_i$ form a cyclically increasing sequence. 
 
\begin{defn} 
Given $L,n$, define $\mathcal{P}_{L,n}$ as the set of ordered set partitions with single marked elements in each block that form a cyclically increasing sequence.
\end{defn}

We have illustrated the following:

\begin{prop}
\label{prop:bij_marked_partitions}
The map associating to a configuration $A$ the data $\left((B_i,b_i)\right)_{i=1,\ldots,L}$ is a bijection from $\mathcal{A}_{L,n}$ to $\mathcal{P}_{L,n}$.
\end{prop}

Indeed, the inverse map consists in letting the particles in row $i$ to occur at positions $B_i$, the particle $\1$ occurring at position $b_i$.
If we define $\mathcal{P}'_{L,n}$ to be the subset of $\mathcal{P}_{L,n}$ where the first marked element is $1$, it is naturally in bijection with $\mathcal{A}'_{L,n}$ by restriction of the above correspondence.

\begin{rem}
The dynamics can be of course transported to this model via this bijection, but the result does not look particularly pleasant or illuminating. Since we do not use it in the following, we leave the explicit description to the interested reader.
\end{rem}

\section{The one-dimensional ASEP}
\label{sec:onedim}

We now recall the one-dimensional asymmetric simple exclusion process (ASEP) on a ring where each particle has different forward and backward jumping rates studied first by Evans~\cite{evans-1996}.

Consider $n$ particles $\1_1,\1_2,\ldots,\1_n$ from left to right (cyclically) on a ring of size $L$ with vacancies denoted by $\2$. Let $\Psi_{L,n}$ be the set of all such configurations. The dynamics
is as follows. Particle $\1_k$ has forward and backward jumping rates $p_k,q_k$ respectively, i.e.
\[
\1_k \, \2 \underset{q_k}{\overset{p_k}{\rightleftharpoons}}
 \2 \1_k,
\] 
and no other transitions.
We note that it is not necessary for the particles $\1$ to be distinguishable. However, it will be convenient from our point of view to suppose that they are.
See \cref{fig:eg-conf} for a configuration and its allowed transitions.
 
\begin{center}
\begin{figure}[!ht]
\includegraphics{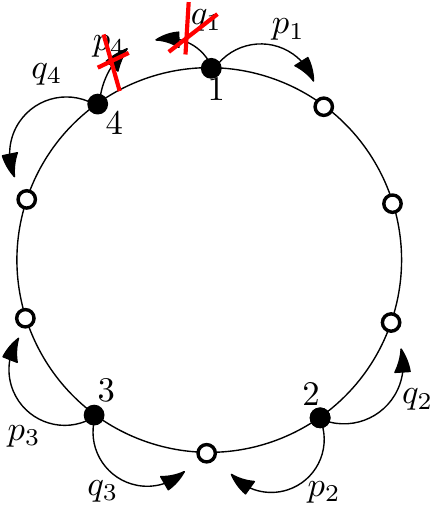}
\caption{The configuration $\tau = \1_1{\2}{\2}{\2}\1_2{\2}\1_3{\2}{\2}\1_4$ for the system with $L=10$ and $n=4$, where the labelling starts from the top and proceeds clockwise.}
\label{fig:eg-conf}
\end{figure}
\end{center}

This ASEP is clearly irreducible if and only if either $\prod_k p_k\neq 0$ or $\prod_kq_k\neq 0$. It has thus a unique stationary distribution, which is explicitly described by Evans~\cite{evans-1996}. 
We will now see how to recover this stationary distribution via ${\Omega}_{L,n}$.

Recall that a {\em projection} or {\em lumping} of a Markov chain is a projection such that the resulting stochastic process is also a Markov chain~\cite[Section 2.3.1]{MarkovMixing}.

\begin{defn}
We define the map $\Pi: \Omega_{L,n} \to \Psi_{L,n}$ as follows.
Given $\omega \in {\Omega}_{L,n}$, replace all ${\2}_i$'s in $\omega$ by $\2$ to obtain $\Pi(\omega) \in \Psi_{L,n}$.
\end{defn}

It is easily checked that the transitions \cref{it:transp1,it:transp2} (resp. \cref{it:transq1,it:transq2}) 
in $\mathcal{A}_{L,n}$ correspond to forward (resp. backward) transitions in $\Omega_{L,n}$.
We thus obtain the following result.

\begin{prop}
\label{prop:lump}
The exclusion process on ${\Omega}_{L,n}$ lumps to the ASEP on $\Psi_{L,n}$ via the map $\Pi$.
\end{prop}

Let $\pi$ denote the stationary distribution of the ASEP on $\Psi_{L,n}$ and $\hat{\pi}$, that for the exclusion process on ${\Omega}_{L,n}$: we will show in 
\cref{prop:irred}
that $\hat{\pi}$ is uniquely defined. It follows then from \cref{prop:lump} that we have the following. 

\begin{cor}
\label{cor:lump}
For every $\psi \in \Psi_{L,n}$,
\[
\pi(\psi) = \sum_{\omega \in \Pi^{-1}(w)} \hat{\pi}(\omega).
\]
\end{cor}

\section{Stationary distribution}
\label{sec:statdist}

We now describe the stationary distribution of the two-dimensional exclusion process on $\mathcal{A}_{L,n}$. 
It will be convenient for some of the proofs to work with the isomorphic multispecies exclusion process on $\Omega_{L,n}$; see \cref{prop:isom}. Let us recall the dynamics of the latter for easy reference:
\begin{align}
\label{transp1bis}
\cdots \1_k| {\2}_i \cdots {\overset{p_k}{\longrightarrow}}&\quad \cdots {\2}_i | \1_k\cdots & \hspace*{-1.5cm} \text{ if }i\neq k,\\
\label{transp2bis}
 \cdots \1_{k-1} C \1_k|{\2}_k \cdots {\overset{p_k}{\longrightarrow}}&\quad \cdots \1_{k-1}{\2}_{k-1} C |\1_k\cdots,\\
\label{transq1bis}
\cdots  {\2}_i|\1_k\cdots {\overset{q_k}{\longrightarrow}}&\quad \cdots \1_k | {\2}_i\cdots & \hspace*{-1.5cm} \text{ if }i\neq k,\\
\label{transq2bis}
 \cdots  {\2}_k |\1_k C \1_{k+1}\cdots {\overset{q_k}{\longrightarrow}}&\quad \cdots \1_k | C {\2}_{k+1}\1_{k+1}\cdots,
\end{align}  
\smallskip

{\noindent\bf Note:} In this section we consider the case where $p_k,q_k>0$ for all $k$. This is slightly less general than the conditions in the previous section. We will extend the results to the more general case in~\cref{sec:qi_zero}.

\subsection{Irreducibility}

Recall that we work under the assumptions $p_k, \allowbreak q_k>0$ for all $k$.

\begin{prop}
\label{prop:irred}
Let $L \geq 1$ and $1 \leq n < L$. The exclusion process on $\mathcal{A}_{L,n}$ and on $\Omega_{L,n}$ is irreducible.
\end{prop}

\begin{proof} 
We will show that starting from any $\tau \in \Omega_{L,n}$, we can reach the special configuration 
\[
\tau_0=(\1_1,\dots,\1_n,{\2}_n,\dots,{\2}_n),
\] 
and conversely.

Fix $\tau \in \Omega_{L,n}$. Say that a configuration $\tau'$ is {\em basic} if for all $k$, the particles between $\1_k$ and $\1_{k+1}$ are only ${\2}_k$'s. We claim that starting with $\tau$ and applying forward transitions of type \eqref{transp1bis} and \eqref{transp2bis}, one can reach a basic configuration. 

To show this, define the \emph{excess} of a particle $\2_i$ in a given configuration to be the integer $k\in\{0,1,\ldots,n-1\}$ such that the nearest particle $\1$ to the left of $\2_i$ is $\1_{i+k}$ (recall that indices for particles are understood modulo $n$). Define the {excess} of a configuration to be the sum of all its particles of type $\2$. For example, for $n=4$, $\1_1 {\2}_3{\2}_3{\2}_4\1_2 {\2}_2\1_3 {\2}_3\1_4 {\2}_1$ from \cref{eg:trans} has excess equal to
\[
(2+2+1) + (0) + (0) + (3) = 8.
\]
The excess of a configuration is zero if and only if it is basic. If the excess is $e>0$, then there exists a contiguous subconfiguration of the form 
\[\1_k\underbrace{{\2}_k\dots {\2}_k}_{t\geq 0}{\2}_i\] with $i\neq k$. By applying $t$ times \eqref{transp2bis} and then \eqref{transp1bis} once, we obtain a configuration has excess $e-1$. The claim is thus proved by induction.

We have thus reached a basic configuration $\tau'$. By continuing to apply transitions of type \eqref{transp2bis}, in order to have the particles $\1_i$ occurring consecutively, which is clearly possible, we can reach new basic configuration that is a cyclic shift of the configuration $\tau_0$. Finally, notice that using \eqref{transp2bis} applied successively with $k=n,n-1,\dots,1$ rotates such  configuration to the right, and so we can eventually reach $\tau_0$ itself.

We now have to prove that from $\tau_0$ one can reach any $\tau\in \Omega_{L,n}$. Equivalently, we have to prove that if one reverses the arrows in the dynamics, one can reach $\tau_0$ from $\tau$ in the resulting graph. In this graph, \eqref{transp2bis} and \eqref{transq1bis} become edges
\begin{align}
\label{transp2bis_reversed} \cdots \1_{k}{\2}_{k} C |\1_{k+1}\cdots {\overset{p_{k+1}}{\longrightarrow}}&\quad \cdots \1_{k} C \1_{k+1}|{\2}_{k+1} \cdots,\\
\label{transq1bis_reversed} \cdots \1_k | {\2}_i\cdots {\overset{q_k}{\longrightarrow}}&\quad \cdots  {\2}_i|\1_k\cdots & \hspace*{-1.5cm} \text{ if }i\neq k,
\end{align} 
Note the formal similarity of the two above transitions with \eqref{transp2bis} and \eqref{transp1bis} respectively. In particular, the rates have changed but this does not affect the analysis since these are all nonzero. 

One can apply the same strategy as in the first part of the proof to show that with the transitions \eqref{transp2bis_reversed}, \eqref{transq1bis_reversed} one can reach $\tau_0$ from $\tau$. 
\end{proof}

By \cref{prop:irred}, the Markov chain has a unique stationary distribution. It is necessary invariant under horizontal translation, since the transition rates have the same property; see \cref{rem:dynamics}(1).

\subsection{Weights}
Let $A \in \mathcal{A}_{L,n}$ be a configuration. For $i=1,\dots,n$, let $b_i$ be the column containing $\1$ in row $i$ of $A$. We have therefore $b_1 < \cdots < b_n$ cyclically.

\begin{defn}
\label{Ck-defi}
Let $k\in\{1,\dots,n\}$. We define $C_k \equiv C_k(A)$ be the(possibly empty) open integer interval $(b_k,b_{k+1}) 
= \{b_k + 1, \dots, b_{k+1} - 1\}$.
\end{defn}

Let $j\in C_k$. We define the weight of a particle $\2$ at position $(i,j)$, $j \in C_k$ to be ($i,k$ are taken to be in $\{1,\ldots,n\}$ here):
\begin{equation}
\label{wt-2}
w_\2(i,k) = \begin{cases}
p_1 \cdots p_{i-1} q_{i+1} \cdots q_k p_{k+1} \cdots p_n & 1 \leq i \leq k, \\
q_1 \cdots q_{k} p_{k+1} \cdots p_{i-1} q_{i+1} \cdots q_n & k < i \leq n.
\end{cases}
\end{equation}

\begin{defn}[Stationary weight] 
The weight $\wt(A)$ of $A \in \mathcal{A}_{L,n}$ is 
\begin{equation}
\label{wt-omega}
\wt(A) = \prod_{k=1}^n \prod_{\substack{j\in C_k \\ A_{i,j}=\2}} w_\2(i,k),
\end{equation}
the product of weights associated to all $\2$'s.
\end{defn}

For example, the weight of the configuration in \cref{eg:trans} and \cref{fig:multiline} is 
\begin{align*}
\wt(\1_1 {\2}_3{\2}_3{\2}_4 \!\1_2 \!{\2}_2 \!\1_3 \! {\2}_3 \! \1_4 \!{\2}_1 \!)
&=
\underbrace{(q_1 p_2 q_4)^2 (q_1 p_2 p_3)}_{C_1}
\underbrace{(p_1 p_3 p_4)}_{C_2}
\underbrace{(p_1 p_2 p_4)}_{C_3}
\underbrace{(q_2 q_3 q_4)}_{C_4} \\
 &= p_1^2 p_2^4 p_3^2 p_4^2 q_1^3 q_2 q_3 q_4^3.
\end{align*}

\begin{rem}
We make a curious observation about the weights in \eqref{wt-2}. The determinant of the matrix formed by these weights has a very nice formula,
\[
\det (w_\2(i,k))_{1 \leq i,k \leq n} = \left( p_1 \cdots p_n - q_1 \cdots q_n \right)^{n-1}.
\]
It can be computed by simple row operations transforming the matrix into a lower triangular matrix. 
The factor on the right hand side will appear frequently in \cref{sec:denscurr} when we calculate the currents.
\end{rem}

\subsection{Stationary distribution}

We can now state the exact form of the stationary distribution of the exclusion process $\mathcal{A}_{L,n}$. Note that we have to assume that all rates $p_k,q_k$ are nonzero here. The special case where some rates vanish is treated in the next section.

\begin{thm}
\label{thm:ssrefined}
Let $L \geq 1$ and $1 \leq n < L$ and suppose $p_k,q_k>0$ for all $1 \leq k \leq n$. 
Then the stationary probability $\hat\pi(A)$ of the configuration $A$ for the exclusion process on $\mathcal{A}_{L,n}$ is proportional to $\wt(A)$.
\end{thm}

\begin{proof}
Since the stationary probabilities are unique by \cref{prop:irred}, it is enough to verify
the balance equation,
\begin{equation}
\label{eq:balance}
\sum_{\tau \in \mathcal{A}_{L,n}} \hat\pi(A)\text{rate}(A \to \tau)
=
\sum_{\tau \in \mathcal{A}_{L,n}} \hat\pi(\tau)\text{rate}(\tau \to A),
\end{equation}

for every configuration $A$. Since all transitions are initiated by particles of type $\1$, it suffices to look at the positions of these particles.
Moreover, every particle of type $\1$ can move to its current location in at most two ways, one from the left and one from the right.

Fix $1 \leq k \leq n$. We will focus on transitions affecting the positions in $C_k$, defined above. If $C_k$ is empty, there cannot be any transitions, either outgoing or incoming, affecting $C_k$.

Suppose $C_k$ is nonempty. Then the outgoing weight of transitions from $A$ is given by $(p_k+q_{k+1}) \wt(A)$. We focus first on the particle of type $\1$ in row $k$. Let $n_k$ be its column. 
The incoming transition that brings $\1$ to this position depends on the row $i$ of the $\2$ in column $n_k + 1$. 
Suppose $i \neq k$. 
Let $A_1$ be the (unique) configuration in the
state space which goes to $A$ with rate $q_k$. Then $A_1$ is obtained from $A$ by switching columns $n_k$ and $n_k +1$. 
In that case,
\[
\frac{\hat{\pi}(A_1)}{\hat{\pi}(A)} = 
\begin{cases}
\frac{\ds p_1 \cdots p_{i-1} q_{i+1} \cdots q_{k-1} p_{k} \cdots p_n}
{\ds p_1 \cdots p_{i-1} q_{i+1} \cdots q_k p_{k+1} \cdots p_n} & i < k, 
\\[0.25cm]
\frac{\ds q_1 \cdots q_{k-1} p_{k} \cdots p_{i-1} q_{i+1} \cdots q_n}
{\ds q_1 \cdots q_{k} p_{k+1} \cdots p_{i-1} q_{i+1} \cdots q_n} & i > k.
\end{cases}.
\]
Thus, $q_k \hat{\pi}(A_1) = p_k \hat{\pi}(A)$.
Suppose the particle of type $\1$ in row $k+1$ is at column $n_{k+1}$.
If $i=k$, then the incoming transition comes from the configuration $A_2$, in which column $n_k + 1$ is moved to column $n_{k+1}$, the particle of type $\2$ is moved to row $k+1$, and all intermediate columns are shifted left. This transition happens with rate $p_{k+1}$.
Then 
\[
\frac{\hat{\pi}(A_2)}{\hat{\pi}(A)} = 
\frac{p_1 \cdots p_{k} p_{k+2} \cdots p_n}
{p_1 \cdots p_{k-1} p_{k+1} \cdots p_n},
\]
and $p_{k+1} \hat{\pi}(A_2) = p_k \hat{\pi}(A)$.
Thus, the incoming weight of the particle of type $\1$ in row $k$ affecting $C_k$ is the same as the outgoing weight.

We now look at the particle of type $\1$ in row $k+1$. 
The incoming transition that brings $\1$ to this position depends on the row $j$ of the $\2$ in column $n_{k+1} - 1$. If $j \neq k+1$, the incoming transition comes from $A_3$, which is obtained from $A$ by switching columns $n_{k+1}$ and $n_{k+1} - 1$, with rate $p_{k+1}$. In that case,
\[
\frac{\hat{\pi}(A_3)}{\hat{\pi}(A)} = 
\begin{cases}
\frac{\ds p_1 \cdots p_{j-1} q_{j+1} \cdots q_{k+1} p_{k+2} \cdots p_n}
{\ds p_1 \cdots p_{j-1} q_{j+1} \cdots q_k p_{k+1} \cdots p_n} & j \leq k, \\[0.25cm]
\frac{\ds q_1 \cdots q_{k+1} p_{k+2} \cdots p_{j-1} q_{j+1} \cdots q_n}
{\ds q_1 \cdots q_{k} p_{k+1} \cdots p_{j-1} q_{j+1} \cdots q_n} & j > k+1.
\end{cases}.
\]
Thus, $p_{k+1} \hat{\pi}(A_3) = q_{k+1} \hat{\pi}(A)$.
If $j = k+1$, then the incoming transition comes from the configuration $A_4$, in which column $n_{k+1} - 1$ is moved to column $n_{k}$, the particle of type $\2$ is moved to row $k$, and all intermediate columns are shifted right. This transition happens with rate $q_{k}$.
Then 
\[
\frac{\hat{\pi}(A_4)}{\hat{\pi}(A)} = 
\frac{q_1 \cdots q_{k-1} q_{k+1} \dots q_n}
{q_1 \cdots q_{k} q_{k+2} \dots q_n},
\]
and $q_{k} \hat{\pi}(A_4) = q_{k+1} \hat{\pi}(A)$.
We have thus matched all the incoming and outgoing transitions affecting $C_k$, and this argument holds for all $k$. Thus, we have proved that the weight function in \eqref{wt-omega} satisfies the master equation.
\end{proof}

\begin{eg}
\label{eg:ss-42}
Here is the set of restricted configurations in $\mathcal{A}'_{4,2}$ from \cref{fig:small_example}, together with their weights:
\begin{center}
\includegraphics[width=\textwidth]{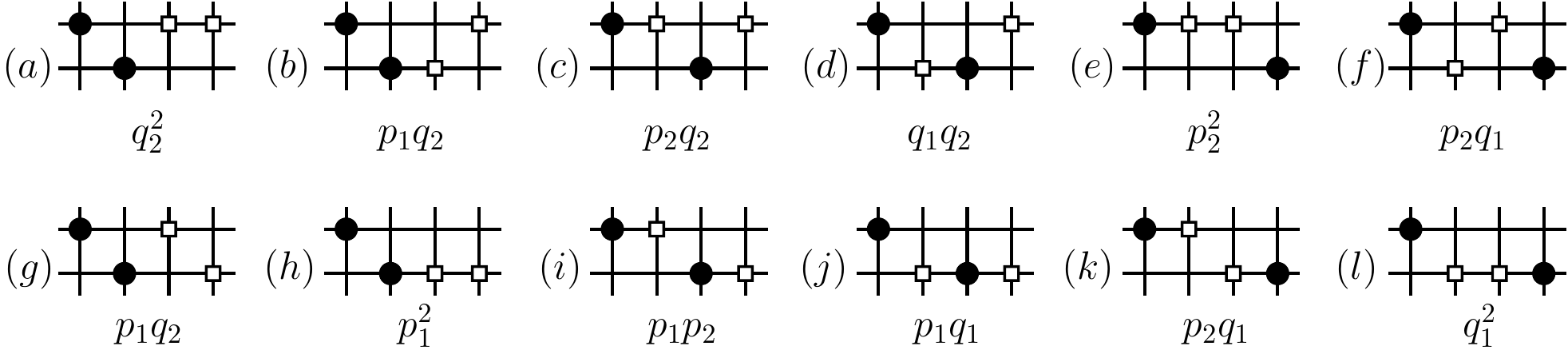}
\end{center}
The balance equation \eqref{eq:balance} can be checked at each of them. Consider the case of configuration (a): Incoming transitions occur from (c) with rate $q_2$ and from a translated version of (d) with rate $q_1$. Outgoing transitions occur with rates $q_1$ and $p_2$, and the total weights match.
\end{eg}

For convenience, we define
\begin{equation}
\label{W-def}
W_\2(k) = \sum_{j=1}^n w_\2(j,k).
\end{equation}

We now apply this result to the one-dimensional model $\Psi_{L,n}$ of \cref{sec:onedim}. For a configuration $\tau \in \Psi_{L,n}$ with $\tau_1 = \1_1$, let $c_i$ count the number of vacancies between $\1_i$ and $\1_{i+1}$. This completely encodes the configuration up to rotation. The proof of the following result is then a simple consequence of \cref{thm:ssrefined} and \cref{cor:lump}.

\begin{cor}[{\cite{evans-1996}}]
\label{cor:ss}
The stationary distribution $\pi$ of the exclusion process on $\Psi_{L,n}$ is given as follows: for any $\tau \in \Psi_{L,n}$ with tuple $(c_1,\ldots,c_n)$ summing to $L-n$, the stationary probability $\pi(\tau)$ is
proportional to $\ds\prod_{k=1}^n (W_\2(k))^{c_k}$.
\end{cor}

The proof in~\cite{evans-1996} uses a matrix ansatz to come up with this product form.

\begin{eg} 
For $n=4$, $W_\2(1)=p_2p_3p_4+q_1p_2p_3+q_4q_1p_2+q_3q_4q_1$ and $W_\2(i)$ for $2 \leq i \leq 4$ are obtained by shifting indices. The stationary probability $\pi(\tau)$ of the configuration $\tau$ in \cref{fig:eg-conf}, which is encoded by the tuple $(3,1,2,0)$, is thus proportional to
$W_\2(1)^3 W_\2(2) W_\2(3)^2$.
\end{eg}

\subsection{Partition function}

The {\em restricted partition function} is defined as
\begin{equation}
Z_{L,n} = \sum_{\substack{A \in \mathcal{A}_{L,n} \\ A_{1,1} = \1}} \wt(A).
\end{equation}
$Z_{L,n}$ is a polynomial in the variables $p_1,\dots,p_n, q_1, \dots, q_n$, homogeneous of degree $L-n$.
By the translation invariance in \cref{rem:dynamics}(1), the full {\em partition function} is $L Z_{L,n}$.
Recall that, for $f(x) = \sum_i a_i x^i$ a polynomial or formal power series in the variable $x$, the notation $[x^i] f(x)$ stands for the coefficient of $x^i$ in $f(x)$, namely $a_i$.

\begin{thm}[{\cite{evans-1996}}]
\label{thm:pf}
The restricted partition function $Z_{L,n}$ is given by:
\[
Z_{L,n} = 
\sum_{\substack{c_1,\ldots,c_n \geq 0\\ c_1+\cdots+c_n=L-n}}
W_\2(1)^{c_1} \cdots W_\2(n)^{c_n}
=[x^{L-n}]\left( \prod_{k=1}^n \frac{1}{1-W_\2(k) x} \right).
\]
\end{thm}

\begin{proof}
The first equality follows from the definition of the partition function. The second one is an immediate consequence of the expansion of the rational function as a series in $x$.
\end{proof}

For example, for $L=4$ and $n=2$, $Z_{4,2}$  is the coefficient of $x^2$ in $(1-(p_1+q_2)x)^{-1}(1-(p_2+q_1)x)^{-1}$, that is
\[
Z_{4,2} = (p_2+q_1)^2+(p_1+q_2) (p_2+q_1)+(p_1+q_2)^2.
\]
It corresponds as expected to the sum of the weights in \cref{eg:ss-42}.

To end this section, we give two special cases that are easy to prove: first we consider the case where the $n$ particles have identical rates, then we consider the case where particles have symmetric jumps. 

\begin{prop}
\label{prop:identical_particles}
If we set $p_i = p$ and $q_i = q$ for all $i$, then
\[
Z_{L,n} = \binom{L-1}{n-1} [n]_{p,q}^{L-n},
\]
where $[n]_{p,q} = p^{n-1} + p^{n-2} q + \cdots + p q^{n-2} + q^{n-1}$.
\end{prop}

Recall that the {\em elementary symmetric polynomial} 
$e_k(x_1,\dots,x_j)$, for $1 \leq k \leq j$ is given by
\begin{equation}
e_k(x_1,\dots,x_j) = \sum_{1 \leq i_1 < i_2 < \cdots < i_k \leq j} 
x_{i_1} x_{i_2} \dots x_{i_k}.
\end{equation}

\begin{prop}
\label{prop:symmetric_particles}
If we set $q_i = p_i$ for all $i$, then
\[
Z_{L,n} = \binom{L-1}{n-1} e_{n-1}(p_1,\dots,p_n)^{L-n}.
\]
\end{prop}

It is somewhat surprising that we obtain a manifestly symmetric function in the $p_i$'s even though a priori we should only expect the partition function to be symmetric under cyclic permutations.

\section{Some totally asymmetric particles} 
\label{sec:qi_zero}

In this section we consider the exclusion process on $\mathcal{A}_{L,n}$, or equivalently $\Omega_{L,n}$, where some parameters $q_i$ are equal to zero. This is not immediately a special case of the results of the previous section. Indeed, in this case, the chain on $\Omega_{L,n}$ is not irreducible\footnote{This can be seen directly, and also from Theorem~\ref{thm:ssrefined}: if some $q_i$ vanishes then certain weights vanish, which cannot happen for the stationary distribution of an ergodic Markov chain.} any more. Hence, we need to modify some of the results of the \cref{sec:statdist}.

\begin{rem} 
The case where some parameters $p_i$ are zero is treated similarly by symmetry. Also, if there exist $i_1,i_2$ such that $p_{i_1}=q_{i_2}=0$, then even the one-dimensional ASEP of \cref{sec:onedim} is not irreducible, and so we are not interested in this case.
\end{rem}

\begin{defn} Let $I\subseteq\{1,\ldots,n\}$ be the set of indices $k$ such that $q_k=0$. We define $\Omega^I \equiv \Omega^I_{L,n}$ to be the subset of states $\tau \in \Omega_{L,n}$ such that $\wt(\tau)\neq 0$.
Similarly, let $\Omega'^{I}$ by the subset of $\Omega^I$ consisting of the states $\tau=(\tau_1,\tau_2,\dots,\tau_L)$ such that $\tau_1 = \1_1$.

\end{defn}

Of course $\Omega^{\emptyset}=\Omega$. Moreover we have
\begin{equation}
\label{eq:Omega_I_intersection}
\Omega^I=\cap_{i\in I}\Omega^{\{i\}}.
\end{equation}
Indeed, this follows immediately from the fact that $\wt(\omega)$ is a monomial in the $q_i$'s.

We thus consider the case where $I$ has cardinality $1$. 
We also assume $I=\{1\}$ without loss of generality because of translational symmetry in \cref{rem:dynamics}(1).

\begin{prop}
\label{prop:ta_description1} 
$\Omega'^{\{1\}}$ is the set of states starting with $\1_1$ such that for any $i$, particles of type $\2_i$ are only allowed to occur to the right of particle $\1_{i}$. 
\end{prop}

\begin{proof}
By the definition of the weight $\wt$, given a configuration $\omega\in\Omega'_{L,n} $, $q_1$ occurs in $\wt(\omega)$ if and only if there is a particle  $\2_i$ occurring to the left of $\1_i$ for a certain $i$. Therefore if $q_1=0$ and $q_i>0$ for $i>1$, $\wt(\omega)$ is nonzero if and only for all $i$, the particles $\2_i$ occur to the right of $\1_i$.
\end{proof}

Note that the running example in \cref{fig:multiline} does not belong to $\Omega'^{\{1\}}$. The above proposition together with \eqref{eq:Omega_I_intersection} implies a characterization of any $\Omega_I$, from which the following lemma can then be directly checked:

\begin{lem} 
$\Omega^I$ is stable under the dynamics of \eqref{transp1bis}--\eqref{transq2bis}, where we naturally exclude the transitions  with $k\in I$ in \eqref{transq1bis},\eqref{transq2bis}.
\end{lem}

We then arrive at the following theorem that extends ~\cref{thm:ssrefined} to the case $I\neq \emptyset$:

\begin{thm}
$\Omega^I$ forms an irreducible Markov chain whose steady state probabilities $\hat\pi(\omega)$ are proportitional to $\wt(\omega)$ for $\omega\in \Omega^I$.
\end{thm}

\begin{proof}
To prove irreducibility, it is enough to pick $I=\{1\}$ thanks to \eqref{eq:Omega_I_intersection} and translational invariance. Irreducibility follows by inspecting the proof in \cref{prop:irred} and checking that it restricts to this case.
The steady state probabilities then follow immediately by verifying the balance equations as in the proof of \cref{thm:ssrefined}, which here also can be restricted. 
\end{proof}

 Now let us consider the correspondence with marked partitions in \cref{sub:partitions}. By restricting the correspondence to $\Omega'^{\{1\}}$ using \cref{prop:ta_description1}, one arrives at the subset of ordered partitions in which the marked element in each block is its smallest element, so we can erase the mark without losing any information. We obtain ordered partitions where the blocks are ordered according to the relative order of their minimum element. Given a standard set partition, there is obviously a unique way to order its blocks in this way and we finally obtain the following result:

\begin{prop}
\label{prop:ta_description2}
$\Omega_{L,n}'^{\{1\}}$ is in bijection with set partitions via the map from \cref{prop:bij_marked_partitions}.
\end{prop}

This explains the occurrence of set partitions in \cite[Theorem 3.7]{ayyer-2020}, which corresponds to the case $p_1>0$, $q_1=0$ and $p_i=q_i=1$ for $i>1$.

\begin{rem}
The generating function of the partition function in the case $q_1>0$ in \cite[Theorem 2.7]{ayyer-2020} can be explained by enumerating marked set partitions. Indeed with these special rates, the weight $\wt(\cdot)$ becomes simple to express and standard methods of enumerative combinatorics give the desired answer. We do not know how to obtain such an exponential generating function in the general case considered in this manuscript.
\end{rem}

Finally, let us describe the totally asymmetric case $I=\{1,\ldots,n\}$.
 Then $\Omega'^{I}$ consists of configurations $\omega$ such that only $\2_k$ can occur between $\1_k$ and $\1_{k+1}$. Note that such a configuration is uniquely determined by the positions of the $\1_j$'s, and its weight is the product over all remaining positions of
$p_1 \cdots p_{k-1} p_{k+1} \cdots p_n=p_1 \cdots p_n/p_k$ if the position is between $\1_k$ and $\1_{k+1}$.
Let the \emph{homogeneous symmetric polynomial} of degree $k$ be defined by
\[
h_k(x_1,\dots,x_j) = \sum_{1 \leq i_1 \leq i_2 \leq \cdots \leq i_k \leq j} 
x_{i_1} x_{i_2} \dots x_{i_k}.
\] 
From the above observation, we obtain the following.

\begin{prop}
\label{prop:Z_totally_asymmetric}
 If $q_i=0$ for $i=1,\ldots,n$,
\[
Z_{L,n}= (p_1\dots p_n)^{L-n}\;
h_{L-n} \left(\frac{1}{p_1},\dots,\frac{1}{p_n} \right).
\]
\end{prop}

\begin{rem}
For the expert, we note that \cref{prop:Z_totally_asymmetric} can also be rewritten in terms of Schur polynomials as
\[
Z_{L,n} = s_{\langle (L-n)^{n-1} \rangle}(p_1,\dots,p_n).
\]
The interested reader can figure out the weight-preserving bijection between $\Omega'^{I}$ and rectangular semistandard Young tableaux that interprets this equality.
\end{rem}

\section{Densities and currents on the torus}
\label{sec:denscurr}

Let $\tau$ (resp $\eta$) denote the occupation variable for $\1$ (resp. $\2$). That is to say, $\tau_{i,j} = 1$ (resp. $\eta_{i,j} = 1$) in a configuration if and only if the site $(i,j)$ is occupied by a $\1$ (resp. $\2$), and otherwise $\tau_{i,j} = 0$ (resp. $\eta_{i,j} = 0$).
We denote expectations in the stationary distribution $\hat{\pi}$ on $\mathcal{A}_{L,n}$ by $\aver{\cdot}_{L,n}$. When $L$ and $n$ are clear from context, we will suppress the subscripts.
Throughout this section, we will assume that all $p_i, q_i > 0$.

\subsection{Densities}

By horizontal translation invariance of the exclusion process on $\mathcal{A}_{L,n}$ in \cref{rem:dynamics}(1), the following is easy to prove.

\begin{prop}
\label{prop:dens1}
The density of $\1$'s is given by
\[
\aver{ \tau_{i,j} } = \frac{1}{L},
\]
for $1 \leq i \leq n, 1 \leq j \leq L$.
The density of $\2$'s satisfy
\[
\aver{ \eta_{i,j} } = \aver{ \eta_{i,j'} },
\]
for $1 \leq i \leq n, 1 \leq j < j' \leq L$.
\end{prop}

The exclusion process is not vertically translation-invariant, but it is if we replace $p_i, q_i$ by $p_{i+1}, q_{i+1,1}$ and down-shift configurations cyclically (see \cref{rem:dynamics}). Using this property, we can show:

\begin{prop}
\label{prop:dens2}
The density of $\2$'s satisfy
\[
\aver{ \eta_{i,1} } \Big\rvert_{
\substack{p_j \to p_{j+1} \\ q_j \to q_{j+1}} \forall j} = \aver{ \eta_{i+1,1} },
\]
for $1 \leq i \leq n$.
\end{prop}

By \cref{prop:dens1} and \cref{prop:dens2}, it suffices to determine $\aver{ \eta_{1,1} } $ in order to compute the densities $\aver{\eta_{i,j}}$ for all $i,j$.
Recall the formula for the weight of a $\2$ in \eqref{wt-2} and $W_\2$ in \eqref{W-def}.

\begin{thm}
\label{thm:dens2}
The density of $\2$'s in position $(1,k)$ is given by
\[
\aver{ \eta_{1,k}} = 
\sum_{i=1}^n w_\2(1,i) \sum_{j=1}^{L-n} W_\2(i)^{j-1} 
\frac{Z_{L-j,n}}{L Z_{L,n}}.
\]
\end{thm}

\begin{proof}
By translation invariance, we can choose any $k$ between $1$ and $L$. So, we let $k=L$. Suppose the nearest $\1$ to its left is in position 
$(i,j)$. Then, by definition, $(1,L) \in C_i$ and we get a contribution of $w_\2(1,i)$ from this $\2$. Since there are $L - n$ columns containing $\2$'s, $j \geq n$ and there are $L-1-j$ columns between the $\1$ in position $(i,j)$ and the $\2$ in position $(1,L)$. By construction, there are no $\1$'s in any of these columns. Therefore, we get a contribution of $W_\2(i)$ from each of these columns. By 
ignoring all the columns after $j$, we obtain a restricted configuration in $\mathcal{A}_{j,n}$, where the normalization is that position $(i,j)$ contains a $\1$. Since the first $j-1$ columns of the configuration can be chosen independent of the columns $j+1$ to $L-1$, we obtain the desired result.
\end{proof}

\begin{eg}
For $L=4$ and $n=2$, we compute the weights of states for which there is a $\2$ at $(1,4)$ from \cref{eg:ss-42} to obtain
\[
2p_2^2 + 2p_1 q_2 + 2p_2 q_1 + 2p_2 q_2 + 2q_2^2 + p_1 p_2 + q_1 q_2,
\]
where we have also considered the rotation of restricted configurations therein.
From the formula in \cref{thm:dens2}, the numerator is
\[
(p_2 + q_2) Z_{3,2} + \left(p_2 (p_2 + q_1) + q_2 (p_1 + q_2) \right) Z_{2,2}.
\]
Plugging in $Z_{3,2} = p_1 + p_2 + q_1 + q_2$ and $Z_{2,2} = 1$, one easily sees that these two expressions are the same.
\end{eg}

\subsection{Currents}

We will compute five kinds of currents. Two of these are analogues of computations in one dimension, namely the currents of $\1$'s and $\2$'s along the horizontal edge $(i,j)$ -- $(i,j+1)$, denoted $J_\1(i,j)$ and $J_\2(i,j)$ respectively. 
The total horizontal current for the $\1$ in row $i$ is then the sum over all $j$ of $J_\1(i,j)$ and will be denoted $J_\1(i)$.
The next two currents are interesting for particles of type $\2$'s since they perform nonlocal motion under transitions \cref{it:transp2} and \cref{it:transq2}.
For any column $j$, we will calculate the motion of $\2$'s passing 
from the left of (and including) column $j$ to the right of it and vice versa. This will be the cumulative horizontal current across edges $(i,j)$ -- $(i,j+1)$ summed over all $i$, denoted $J_\2^{\text{h}}(j)$.
Similarly, for any row $k$, the motion of $\2$'s passing below (and including row $i+1$) to above it and vice versa gives the cumulative vertical current across all vertical edges from row $i$ to row $i+1$, denoted $J_\2^{\text{v}}(i)$.
See \cref{fig:current-types} for an illustration of these currents.

\begin{center}
\begin{figure}[!ht]
\includegraphics[width=0.8\textwidth]{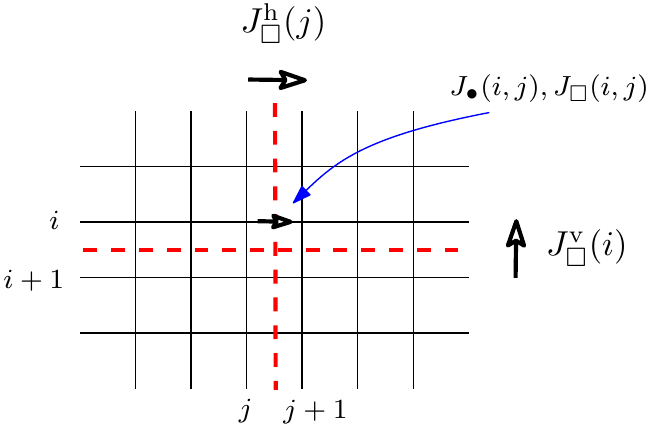}
\caption{An illustration of the types of currents being considered here.}
\label{fig:current-types}
\end{figure}
\end{center}

The following property of the weights will come in useful in the computation of currents.

\begin{lem}
\label{lem:2wts}
The weights associated to $\2$'s satisfy
\[
p_k w_\2(i,k) - q_k w_\2(i,k-1) = 
\begin{cases}
p_1 \cdots p_n - q_1 \cdots q_n & \text{if }i = k, \\
0 & \text{if }i \neq k.
\end{cases}
\]
\end{lem}

\begin{proof}
From \eqref{wt-2}, if $i \leq k-1$, the left hand side is
\[
p_k \big( p_1 \cdots p_{i-1} q_{i+1} \cdots q_k p_{k+1} \cdots p_n \big)
- q_k \big( p_1 \cdots p_{i-1} q_{i+1} \cdots q_{k-1} p_{k} \cdots p_n \big),
\]
which is $0$. Similarly, if $i > k$, the left hand side is
\[
p_k \big( q_1 \cdots q_{k} p_{k+1} \dots p_{i-1} q_{i+1} \dots q_n \big)
- q_k \big( q_1 \cdots q_{k-1} p_{k} \dots p_{i-1} q_{i+1} \dots q_n \big),
\]
which is again $0$. If $i=k$, we get
\[
p_k \big( p_1 \cdots p_{k-1} p_{k+1} \cdots p_n \big)
- q_k \big( q_1 \cdots q_{k-1} q_{k+1} \cdots q_n \big),
\]
as desired.
\end{proof}

We now consider the currents of $\1$'s. Since these only travel horizontally, we can only talk about horizontal currents for these. 
We denote by $J_\1(i,j)$ the current for the particle of type $\1$ between sites $(i,j)$ and $(i,j+1)$. 
More precisely, this is the number of $\1$'s going from $(i,j)$ to $(i,j+1)$ minus the number of $\1$'s going from $(i,j+1)$ to $(i,j)$ per unit time, in the large time limit. By particle conservation, this is independent of $j$. In terms of the stationary distribution, 
it is given by
\[
J_\1(i,j) = p_i \left\langle \tau_{i,j} \sum_{k=1}^n \eta_{k,j+1} \right\rangle
- q_i \left\langle \tau_{i,j+1} \sum_{k=1}^n \eta_{k,j} \right\rangle.
\]

\begin{thm}[{\cite[Equation (5)]{evans-1996}}]
\label{thm:curr1}
For $1 \leq i \leq n$ and $1 \leq j \leq L$, we have
\[
J_\1(i,j) = (p_1 \cdots p_n - q_1 \cdots q_n) 
\frac{Z_{L-1,n}}{L Z_{L,n}}.
\]
\end{thm}

\begin{rem}
We note that the formula for the current in \cite[Equation (5)]{evans-1996} may look superficially different to that in \cref{thm:curr1}.
In particular, it does not have the prefactor $p_1 \cdots p_n - q_1 \cdots q_n$. However, it is actually the same formula, the difference stemming from a different normalization for $Z_{L,n}$ in \cite{evans-1996}.
\end{rem}

We give an alternate proof of \cref{thm:curr1} directly using the exclusion process on the torus.

\begin{proof}
Particle $\1$ moves from site $(i,j)$ to $(i,j+1)$ with rate $p_i$ if there is no $\1$ in the $j+1$'th column, or equivalently, if there is a $\2$ in the $j+1$'th column. Similarly, the $\1$ moves from site $(i,j+1)$ to $(i,j)$ with rate $q_i$ if there is a $\2$ in the $j$'th column. 
Therefore, the current between sites $(i,j)$ and $(i,j+1)$ 
in the stationary distribution depends only on the quantity
\[
\sum_{j=1 }^n \big( p_i w_\2(j,i) - q_i w_\2(j,i-1) \big).
\]
By \cref{lem:2wts}, this is equal to $p_1 \cdots p_n - q_1 \cdots q_n$, which is independent of the configuration.
The weight of the rest of the configuration is the same as if these two columns ($j$ and $j+1$) were deleted, and instead a single column with a $\1$ at position $(i,j)$ was added. Therefore, summing over all configurations gives $Z_{L-1,n}$, completing the proof.
\end{proof}

\begin{rem}
It is instructive to compute this formula in the special case $p_i = p$ and $q_i = q$ for all $i$. From \cref{thm:curr1} and \cref{prop:identical_particles}, the horizontal current of $\1$'s is easily calculated to be
\[
J_\1(1,1) = (p^n-q^n) \frac{L-n}{[n]_{p,q} L(L-1)} = (p-q) \frac{L-n}{L(L-1)}.
\]
Since the current in the one-dimensional ASEP defined in \cref{sec:onedim} is the same, we can calculate it there as well.
Note that in this case the uniform distribution on $\Psi_{L,n}$ is clearly stationary. Then, by a direct calculation, we get
\[
(p-q)\frac{(L-2)!/(L-n-1)!}{L!/(L-n)!},
\]
which gives the same result as above.
\end{rem}

\begin{eg}
We compute the current $J_\1(1,1)$ for $L = 4$ and $n = 2$ using the weights in \cref{eg:ss-42}. The numerator turns out to be
\[
p_1 \left( (p_2 + q_1)^2 + (p_1+q_2) (p_2+q_1) \right)
- q_1 \left( (p_1+q_2) (p_2+q_1) + (p_1+q_2)^2 \right) \!,
\]
which simplifies to 
\[
(p_1 p_2 - q_1 q_2) (p_1 + p_2 + q_1 + q_2),
\]
as expected from \cref{thm:curr1}.
\end{eg}

The total horizontal current $J_\1(i)$ can now be computed directly from 
\cref{thm:curr1}.
\begin{cor}
\label{cor:curr1}
For $1 \leq i \leq n$, we have
\[
J_\1(i) = (p_1 \cdots p_n - q_1 \cdots q_n) 
\frac{Z_{L-1,n}}{Z_{L,n}}.
\]
\end{cor}

We then deduce the horizontal current of $\2$'s crossing columns $j$ and $j+1$, denoted by $J_\2^{\text{h}}(j)$.

\begin{cor}
\label{cor:curr2-h}
For any $j \in [L]$, 
\[
J_\2^{\text{h}}(j) = -n (p_1 \cdots p_n - q_1 \cdots q_n) 
\frac{Z_{L-1,n}}{L Z_{L,n}}.
\]
\end{cor}

\begin{proof}
From the fact that the two-dimensional ASEP projects onto the one-dimensional ASEP in \cref{prop:lump}, $J_\2^{\text{h}}(j)$ is the same as the current of $\2$'s in the one-dimensional ASEP between sites $j$ and $j+1$. Clearly, this is the opposite of the total current of $\1$'s. Since our particles are distinguishable in the two-dimensional ASEP, we must sum over $J_\1(i,j)$ for all $i$ using \cref{thm:curr1}. This then proves the result.
\end{proof}

In terms of the stationary distribution, the current of $\2$'s between rows $i$ and $i+1$ is given by
\[
J_\2^{\text{v}}(i) = \sum_{j=1}^L \big( p_i \aver{\eta_{i,j} \tau_{i,j-1} }
- q_i \aver{ \eta_{i,j} \tau_{i,j+1} } \big).
\]
We will define the {\em upward current} $J_\2^{\text{v}+}(i)$ from row $i$ to $i+1$ and the {\em downward current} $J_\2^{\text{v}-}(i)$ from row $i$ to $i-1$ by keeping track of one-sided motions, namely
\begin{equation}
\label{curr2+-}
J_\2^{\text{v}+}(i) = \sum_{j=1}^L p_i \aver{\eta_{i,j} \tau_{i,j-1} }, \quad
J_\2^{\text{v}-}(i) = \sum_{j=1}^L q_i \aver{ \eta_{i,j} \tau_{i,j+1}}.
\end{equation}

\begin{thm}
\label{thm:curr2-v}
We have
\[
J_\2^{\text{v}+}(i) = p_1 \cdots p_n \frac{Z_{L-1,n}}{Z_{L,n}} 
\quad \text{and} \quad
J_\2^{\text{v}-}(i) = q_1 \cdots q_n \frac{Z_{L-1,n}}{Z_{L,n}}. 
\]
In particular, both are independent of $i$.
\end{thm}

\begin{proof}
A particle $\2$ at site $(i,j)$ can only move upwards if there is a $\1$ at site $(i,j-1)$, and if so, this happens with rate $p_i$. In that case, $(i,j) \in C_i$ and $w_\2(i,i) = p_1 \cdots p_{i-1} p_{i+1} \dots p_n$ by \eqref{wt-2}, which is independent of the configuration. Arguing as in the proof of \cref{thm:curr1}, the weight of the rest of the configuration is the same as if columns $j-1$ and $j$ are replaced by a single column with a $\1$ at position $(k,j-1)$. 
Summing over all such configurations gives $Z_{L-1,n}$ and proves the formula for $J_\2^{\text{v}+}$. The argument for $J_\2^{\text{v}-}$ is entirely analogous.
\end{proof}

\begin{eg}
We compute these currents for $L = 4$ and $n = 2$ using the weights in \cref{eg:ss-42}:
Starting from \eqref{curr2+-}, the contributions to $J_\2^{\text{v}+}(1)$ come from configurations (b), (h), (i) and (j) and sum to
\[
p_2 \left( p_1 q_2 + p_1^2 + p_1 p_2 + p_1 q_1 \right).
\]
Similarly, the contributions to $J_\2^{\text{v}-}(1)$
come from configurations (a), (b), (c) and (d) and give
\[
q_1 \left( q_2^2 + p_1 q_2 + p_2 q_2 + q_1 q_2 \right).
\]
We also have to consider all $4$ rotations in the current computation and then this matches with \cref{thm:curr2-v}.
\end{eg}

From \cref{thm:curr2-v}, we immediately obtain the vertical current of $\2$'s.

\begin{cor}
\label{cor:curr2-v}
The vertical current of $\2$'s is the same as the horizontal current of $\1$'s, i.e.
\[
J_\2^{\text{v}}(i) = J_\1(i) = (p_1 \cdots p_n - q_1 \cdots q_n) 
\frac{Z_{L-1,n}}{Z_{L,n}}.
\]
\end{cor}

The equality of the horizontal current of the $\1$'s in \cref{thm:curr1} and the vertical current of the $\2$'s in \cref{cor:curr2-v} is a manifestly two-dimensional phenomenon. We call this the \emph{Scott Russell phenomenon}, for the linkage named after him shown in \cref{fig:linkage} which translates linear motion in one direction to that in a perpendicular direction. 
We have proved the equivalence of these currents combinatorially for the stationary distribution. It would be interesting to understand this out of equilibrium as well.\smallskip

For the sake of completeness, we compute the current of $\2$'s across a horizontal edge.
For the fixed horizontal edge between sites $(i,j)$ and $(i,j+1)$, $J_\2(i)$ is the current of $\2$'s along that edge, just as we defined the current for particles of type $\1$. 
In terms of the stationary distribution, this is given by
\[
J_\2^{\text{h}}(i) = \left\langle \eta_{i,j} \sum_{\substack{k=1 \\ k \neq i}}^n q_k \tau_{k,j+1} \right\rangle
- \left\langle \eta_{i,j+1} \sum_{\substack{k=1 \\ k \neq i}}^n p_k \tau_{k,j} \right\rangle.
\]

\begin{thm}
\label{thm:curr2-h}
Across a horizontal edge at row $i$, the current of $\2$'s across that edge is given by
\[
J_\2^{\text{h}}(i) = (p_1 \cdots p_n - q_1 \cdots q_n) \frac{\aver{\eta_{i,1}}_{L-1,n} (L-1)Z_{L-1,n}}{L Z_{L,n}},
\]
where the density of $\2$'s on the right hand side is calculated in $\mathcal{A}_{L-1,n}$.
\end{thm}

\begin{proof}
There are two types of ways a $\2$ at position $(i,j)$ can cross a horizontal edge. The first is if there is a $\1$ at $(k,j-1)$ or $(k,j+1)$, $k \neq i$ and a transition of type \cref{it:transp1} or \cref{it:transq1} takes place. The second is if this $\2$ is between $\1$'s at rows $k-1$ and $k$ and a transition of type \cref{it:transp2} or \cref{it:transq2} takes place. We will first show that the net current from the first type of transitions is zero.

For the first type, the $\2$ at position $(i,j)$ can move to $(i,j+1)$ is if there is a $\1$ in column $j+1$ which is not in row $j$ and which moves backward. If this $\1$ belongs to the $k$'th row, then that $\2$ belongs to $C_{k-1}$, and this transition happens with rate $q_k$. Similarly, the backward current happens with rate $p_k$ if there is a $\1$ in the $j$'th column and row $i \neq k$, in which case that $\2$ belongs to $C_k$.
Both these transitions are clearly unaffected by the configurations in columns other than $j$ and $j+1$. Therefore, using \eqref{wt-2}, we need to compute
\[
\sum_{\substack{ k=1 \\ k \neq i}}^n \big( q_k w_\2(i,k-1) -
 p_k w_\2(i,k) \big).
\]
By \cref{lem:2wts}, this is zero.

For the second type, assume that the $\2$ at position $(i,j)$ is between $\1$'s at positions $(k-1,a)$ and $(k,b)$. This $\2$ moves horizontally to the right with rate $p_k$ if there is also a $\2$ at position $(k,b+1)$. Similarly, it moves horizontally to the left with rate $q_{k-1}$ if there is also a $\2$ at position $(k-1,a-1)$. We can now make a bijection between  configurations contributing to the rightward and leftward movement of the $\2$ at $(i,j)$ as shown in \cref{fig:bijection-curr2}.
The positions of the $\2$'s in columns $a+1,\dots,b-1$ in the configuration on the left move right by one step as we go to the configuration on the right; all other positions outside this view are unchanged. 

\begin{center}
\begin{figure}[h!]
\includegraphics[scale=0.65]{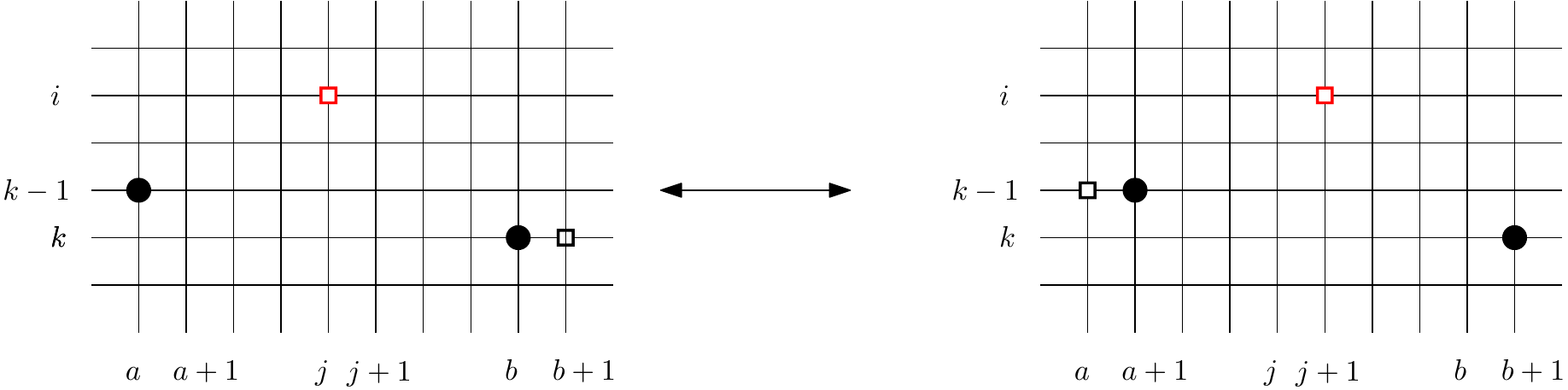}
\caption{A bijection between contributions to the right and left current of the $\2$ at $(i,j)$.}
\label{fig:bijection-curr2}
\end{figure}
\end{center}

We now compute the difference in the forward and backward contributions in this pair of configurations. The only weight that changes is that of the $\2$ at $(k,b+1)$ on the left and $(k-1,a)$ on the right of \cref{fig:bijection-curr2}.
Leaving aside the other weights for the moment, we get
\[
p_k w_\2(k,k) - q_{k-1} w_\2(k-1,k-2) = p_1 \cdots p_n - q_1 \cdots q_n.
\]
We now complete the proof using ideas similar to that of \cref{thm:dens2}.
We can delete the column $b+1$ on the left or $a$ on the right to obtain the same configuration in $\mathcal{A}_{L-1,n}$ where a $\2$ is at position $(i,j)$. Summing over all such configurations gives us 
$\aver{\eta_{i,1}}_{L-1,n} (L-1)Z_{L-1,n}$, proving the formula.
\end{proof}

\section*{Acknowledgements}

This project was initiated at the occasion of the S\'eminaire de combinatoire de Lyon \`a l'ENS, which is supported by Labex MILYON/ANR-10-LABX-0070.
The first author (AA) was partially supported by the UGC Centre for Advanced Studies and by Department of Science and Technology grant EMR/2016/006624. The second author (PN) was supported by the French ANR grant COMBIN\'E (19-CE48-0011).

\bibliographystyle{alpha}
\bibliography{ssep}

\end{document}